\documentclass{amsart}

\usepackage[dvipdfmx]{graphicx}
\usepackage{amsmath}
\usepackage{amssymb}
\usepackage{amsthm}

\usepackage{tabmacD}

\def\CC{\mathord{\mathbb{C}}}

\def\ZZ{\mathord{\mathbb{Z}}}
\def\QQ{\mathord{\mathbb{Q}}}

\def\color#1{}

\title[Grothendieck polynomials and the boson-fermion correspondence]{Grothendieck polynomials\\ and the boson-fermion correspondence}
\author{Shinsuke Iwao}
\address{Department of Mathematics, Tokai University, 4-1-1, Kitakaname, Hiratsuka, Kanagawa 259-1292, Japan.}
\email{iwao@tokai.ac.jp}
\date{\today}

\newtheorem{thm}{Theorem}[section]
\newtheorem{prop}[thm]{Proposition}
\newtheorem{lemma}[thm]{Lemma}
\newtheorem{defi}[thm]{Definition}
\newtheorem{example}[thm]{Example}
\newtheorem{rem}{Remark}[section]

\newtheorem{cor}[thm]{Corollary}

\def\CC{\mathord{\mathbb{C}}}
\def\QQ{\mathord{\mathbb{Q}}}

\def\ZZ{\mathord{\mathbb{Z}}}

\def\oG{\overline{G}}

\newcommand{\Tableau}[1]{%
\def\emp{\emptyset\bl}
\def\maru##1{\scalebox{1.3}{\lower 7pt\hbox{\textcircled{\scriptsize$##1$}}}\bl}%
\def\tama##1{\scalebox{1.3}{\lower 7pt\hbox{\textcircled{\scriptsize$##1$}}}}%
{\text{\tableau[sY]{#1}}}%
}%
\newcommand{\bm}[1]{\mbox{\boldmath{$#1$}}}

\makeatletter

\def\h@iti#1{%
\bgroup%
\let\\=\cr%
\def\r{\hbox to 0pt{$\to$}}%
\def\d{\vbox to 0pt{\hbox to 0pt{$\searrow$}}}%
\vcenter{\tabskip=0pt\halign{&$##$\hspace{10pt}\cr#1\crcr}}%
\egroup%
}

\newcommand{\haiti}[1]{\hspace{4.5pt}\h@iti{#1}\hspace{-4.5pt}}

\makeatother

\def\sayou#1{%
\bgroup%
\let\\=\cr%
\def\+{\hbox{\lower 5pt\hbox{\scalebox{3.5}{$+$}}}}%
\def\-{\hbox{\lower 5pt\hbox{\scalebox{3.5}{$+$}\lower -2pt\hbox to 0pt{\hspace{-19.9pt}\scalebox{2.5}{$\circ$}}}}}%
\vcenter{\tabskip=0pt\halign{&\hfil$##$\hfil\cr#1\crcr}}%
\egroup%
}

\def\ket#1{\vert #1 \rangle}
\def\bra#1{\langle #1 \vert}

\def\zet#1{\left\vert{#1}\right\vert}
\def\ee{\mathbf{e}}

\begin{document}

\begin{abstract}
In this paper we study algebraic and combinatorial properties of symmetric Grothendieck polynomials and their dual polynomials by means of the boson-fermion correspondence.
We show that these symmetric functions can be expressed as a vacuum expectation value of some operator that is written in terms of free-fermions.
By using the free-fermionic expressions, we obtain alternative proofs of determinantal formulas and Pieri type formulas.

\smallskip
\noindent \textbf{Keywords.} Grothendieck polynomials, boson-fermion correspondence, free fermions
\end{abstract}
\maketitle

\section{Introduction}\label{sec:intro}

Grothendieck polynomials \cite{lascoux1982structure,lascoux1983symmetry} are $K$-theoretic versions of Schubert polynomials that represent a Schubert variety in the $K$-theory of the flag variety.
If it represents a Schubert class indexed by a Grassmannian permutation~\cite{buch2002littlewood} (see also \cite[\S 10.6]{fulton_1996}), a Grothendieck polynomial is a symmetric polynomial in finitely many variables.
Such symmetric Grothendieck polynomials are seen as $K$-theoretic analogs of Schur polynomials~\cite{macdonald1998symmetric}.

Let $\lambda=(\lambda_1\geq \dots\geq \lambda_\ell>0)$ be a partition of a natural number, and $\beta$ be a parameter.
A ($\beta$-) \textit{Grothendieck polynomial}\footnote{
The polynomial $G_\lambda(x_1,\dots,x_n)$ is usually called the \textit{$\beta$-Grothendieck polynomial}, which is a deformation of the ordinary Grothendieck polynomial introduced by Fomin-Kirillov~\cite{fomin1994grothendieck}.
The $\beta$-Grothendieck polynomial reduces to the Schur polynomial $s_\lambda(x_1,\dots,x_n)$ when $\beta=0$, and to the ordinary Grothendieck polynomial when $\beta=-1$. 
We will drop the term ``$\beta$-'' throughout the paper to simplify the notation.
}
$G_\lambda(x_1,\dots,x_n)$ ($\ell\leq n$) is a symmetric polynomial in $x_1,\dots,x_n$ that is expressed as follows~\cite{IKEDA201322,kirillov2016some,yeliussizov2017duality}: 
\begin{equation}\label{eq:ratio}
G_\lambda(x_1,\dots,x_n)=
\frac{\det\left(
x_i^{\lambda_j+n-j}(1+\beta x_i)^{j-1}
\right)_{1\leq i,j\leq n}
}{
\prod_{1\leq i<j\leq n}(x_i-x_j)
}.
\end{equation}
There also exists the Jacobi-Trudi type identity \cite{kirillov2016some,lenart2000combinatorial}:
\begin{equation}\label{eq:Jacobi-Trudi}
G_\lambda(x_1,\dots,x_n)=
\det
\left(
\sum_{m=0}^{\infty}
\binom{i-1}{m}
\beta^{m}
h_{\lambda_i-i+j+m}(x_1,\dots,x_n)
\right)_{1\leq i,j\leq n},
\end{equation}
where $h_i(x_1,\dots,x_n)$ is the $i$-th complete symmetric polynomial.

Recently, many authors have been studying connections between these ``$K$-theoretic polynomials'' and the theory of classical/quantum integrable systems.
In \cite{motegi2013vertex,motegi2014k}, Motegi-Sakai showed that Grothendieck polynomials (and their generalizations) are derived from calculations of a wave function of quantum integrable systems such as TASEP and melting crystals. 
Nagai and the author of this paper~\cite{iwao2018discrete} have reported that some class of dual stable Grothendieck polynomials are naturally obtained from tau functions of the relativistic Toda equation with unipotent eigenvalues.

The purpose of this paper is to present a new characterization of $G_\lambda(x_1,\dots,x_n)$ by means of the boson-fermion correspondence (see, for example, \cite{jimbo1983solitons,kac2013bombay}), which is a powerful algebraic tool in various fields such as symmetric polynomial theory, mathematical physics, integrable systems, \textit{etc.}
We show that the stable symmetric Grothendieck polynomial $G_\lambda(x_1,x_2,\dots)$,
that is, an infinite series of symmetric functions with $G_\lambda(x_1,\dots,x_n,0,0,\dots)=G_\lambda(x_1,\dots,x_n)$, 
can be expressed as a vacuum expectation value of some operator given in terms of free-fermions (Theorem \ref{thm:main}).
By using this expression, we derive a similar characterization of the dual  stable Grothendieck polynomial (Section \ref{sec:gr}).

As an application of our presentation, we give new proofs of the following results, which have been given by previous researches:
\begin{enumerate}
\def\labelenumi{(\arabic{enumi}).}
\item ``Another'' determinantal formula for Grothendieck polynomials (Proposition \ref{prop:another_determinant_G}).
This is a special case of the results by Hudson-Ikeda-Matsumura-Naruse~\cite{HUDSON2017115}.
\item A determinantal formula for dual stable Grothendieck polynomials (Proposition \ref{prop:determinant_g}), which was originally given by \cite{lascoux2014finite,shimozono2011}.
\item $G_\lambda(x)$ expansion of symmetric polynomials of the form $s_\lambda(x)G_\mu(x)$ (Proposition \ref{prop:product_sG}).
\item Pieri type formulas for dual stable Grothendieck polynomials (Section \ref{sec:app2}).
Items (3--4) are special cases of the results given by Yeliussizov~\cite{yeliussizov2017duality}.
\end{enumerate}


\subsection{Organization of the paper}

In Section \ref{sec:free_fermions}, we first give a brief review of the theory of free fermions (\S \ref{sec:prelim}--\S \ref{sec:boson-fermion}).
Then we introduce two new operators $e^{\Theta}$ and $e^{\theta}$  and some simple lemmas in \S \ref{sec:etheta_eTheta}.
In Section \ref{sec:Gr}, we present a free-fermionic presentation of stable Grothendieck polynomials.
For this, it is useful to consider a symmetric function $G^r_\lambda(x)$ (\S \ref{sec:def_of_Gr}), which is ``sufficiently near'' to the Grothendieck polynomial $G_\lambda(x_1,\dots,x_r)$ (see Proposition \ref{prop:G_and_Gr}).
We discuss about the ``stable limit'' of the sequence $G^1_\lambda(x),G^2_\lambda(x),\dots$.
Since the sequence itself is not stable, the limit ``$\lim\limits_{r\to \infty}G^r_\lambda(x)$'' fails to be contained in $\Lambda$, the ring of symmetric functions.
However, the limit will be defined properly in some completed space $\widehat{\Lambda}\supset \Lambda$ 
(see \S\ref{sec:completion}).
We show that the limit $\lim\limits_{r\to \infty}G^r_\lambda(x)$ is expressed as a vacuum expectation value of a certain operator written in free-fermions, and is equal to the stable Grothendieck polynomial.
``Another'' determinantal formula for the stable Grothendieck polynomials is shown in \S \ref{sec:another_determinantial_formula}. 
In Section \ref{sec:gr}, we obtain a similar free-fermionic presentation of dual stable Grothendieck polynomials.
Their determinantal formula is also given (\S\ref{sec:determinant_g}).

In Sections \ref{sec:app1} and \ref{sec:app2}, we discuss Pieri type formulas for $K$-theoretic polynomials.
By using our free-fermionic presentations, we define an action of non-commutative Schur polynomials~\cite{fomin1998noncommutative} on the dense linear subspace of $\widehat{\Lambda}$ spanned by $\{G_\lambda(x)\,\vert\,\lambda \mbox{ : partition} \}$.
We also define their action on the subspace spanned by the family $\{g_\lambda(x)\,\vert\,\lambda \mbox{ : partition} \}$, where $g_\lambda(x)$ is the dual stable Grothendieck polynomial.
As a result, we derive the $G_\lambda(x)$ expansion of symmetric polynomials of the form $s_\lambda(x)G_\mu(x)$ (Proposition \ref{prop:product_sG}) and Pieri type formulas for dual stable Grothendieck polynomials (Section \ref{sec:app2}).

\section{Free fermions}\label{sec:free_fermions}

\subsection{Preliminaries}\label{sec:prelim}


Let $k$ be a field of characteristic $0$.
(We will put $k=\CC(\beta)$ in the sequel.)
We consider a $k$-algebra $\mathcal{A}$ generated by \textit{free fermions} $\psi_n$, $\psi_n^\ast$ ($n\in \ZZ$) with the following anti-commutative relations:
\begin{equation}\label{eq:free-fermions-relation}
[\psi_m,\psi_n]_+=[\psi^\ast_m,\psi^\ast_n]_+=0,\qquad
[\psi_m,\psi^\ast_n]_+=\delta_{m,n},
\end{equation}
where $[A,B]_+=AB+BA$.

Let $\ket{0}$, $\bra{0}$ be \textit{vacuum vectors} that satisfy
\[
\psi_m\ket{0}=\psi^\ast_n\ket{0}=0,\quad
\bra{0}\psi_n=\bra{0}\psi^\ast_m=0,\qquad m< 0,\ n\geq 0.
\]
The \textit{Fock space} (over $k$) is the $k$-space $\mathcal{F}$ that is generated by vectors of the form
\begin{equation}\label{eq:elementary-vactors}
\psi_{n_1}\psi_{n_2}\cdots \psi_{n_r}\psi^\ast_{m_1}\psi^\ast_{m_2}\cdots \psi^\ast_{m_s}\ket{0},\
(r,s\geq 0,\ n_1>\dots>n_r\geq 0>m_s>\dots>m_1).
\end{equation}
We also consider the $k$-space $\mathcal{F}^\ast$ that is generated by vectors
\begin{equation}\label{eq:elementary-vetcors-dual}
\bra{0}\psi_{m_s}\cdots \psi_{m_2}\psi_{m_1}\psi^\ast_{n_r}\cdots \psi^\ast_{n_2}\psi^\ast_{n_1},\
(r,s\geq 0,\ n_1>\dots>n_r\geq 0>m_s>\dots>m_1).
\end{equation}

The vectors of the form \eqref{eq:elementary-vactors} are linearly independent over $k$.
This fact can be checked by identifying $\mathcal{F}$ with an \textit{infinite wedge presentation of a Clifford algebra} ~\cite[\S 4 and \S 5.2]{kac2013bombay}. 
See Remark \ref{rem:added} below.
Similarly, the vectors of the form \eqref{eq:elementary-vetcors-dual} are proved to be linearly independent.

Using the anti-commutative relations \eqref{eq:free-fermions-relation} repeatedly, we find $\ket{v}\in \mathcal{F}\Rightarrow \psi_n\ket{v}\in \mathcal{F}$ and $\psi^\ast_n\ket{v}\in \mathcal{F}$.
Hence $\mathcal{F}$ is a left $\mathcal{A}$-module.
We can also check that $\mathcal{F}^\ast$ is a right $\mathcal{A}$-module.

\begin{rem}\label{rem:added}
Let $V=\bigoplus_{i\in \ZZ}k\cdot v_i$ be a $k$-space with a fixed basis $\{v_i\,|\,i\in \ZZ \}$ and $\bigwedge^\infty V$ be the infinite wedge space.
For $m\in \ZZ$, consider the subspace $F^{(m)}\subset \bigwedge^\infty V$ generated by all vectors of the form 
\begin{equation}\label{eq:basis}
v_{i_1}\wedge v_{i_2}\wedge \cdots \quad
(i_1>i_2>\cdots,\  i_k=-k+m \mbox{ for } k\gg 1).
\end{equation}
Set $F:=\bigoplus_{m\in \ZZ}F^{(m)}$.
We can define the actions of $\psi_j,\psi^\ast_j$ on $F$
\cite[\S 5.2]{kac2013bombay} by
\begin{equation}\label{eq:fermion_action}
\begin{aligned}
&\psi_j(v_{i_1}\wedge v_{i_2}\wedge\cdots)=v_j\wedge v_{i_1}\wedge v_{i_2}\wedge\cdots,\\
&\begin{aligned}
\psi^\ast_j(v_{i_1}\wedge v_{i_2}\wedge\cdots)= & 
f(v_{i_1}) v_{i_2}\wedge v_{i_3}\wedge v_{i_4}\wedge \cdots
-
f(v_{i_2}) v_{i_1}\wedge v_{i_3}\wedge v_{i_4}\wedge\cdots\\
&+
f(v_{i_3}) v_{i_1}\wedge v_{i_2}\wedge v_{i_4}\wedge\cdots
-\cdots,
\end{aligned}
\end{aligned}
\end{equation}
where $f:V\to k$ is the $k$-linear map that sends $v_j\mapsto 1$ and $v_k\mapsto 0$ $(k\neq j)$.
We can check that they satisfy the relation \eqref{eq:free-fermions-relation}.
By sending the vacuum vector $\ket{0}$ to the element $v_{-1}\wedge v_{-2}\wedge v_{-3}\wedge \cdots \in F$, we obtain a left $\mathcal{A}$-module homomorphism $\mathcal{F}\to F$, which is in fact an isomorphism.
Rigorous statements and proofs of these facts can be found in the textbooks \cite[\S 5]{kac2013bombay} and \cite[\S 4]{miwa2012solitons}.
The review article~\cite[\S 2]{alexandrov2013free} by Alexandrov and Zabrobin is helpful to understand the free-fermion formalism.
\end{rem}

For an integer $m$, we define the \textit{shifted vacuum vectors} $\ket{m}\in \mathcal{F}$ and $\bra{m}\in\mathcal{F}^\ast$ by
\[
\ket{m}=
\begin{cases}
\psi_{m-1}\psi_{m-2}\cdots \psi_0\ket{0}, & m\geq 0,\\
\psi^\ast_{m} \cdots\psi^\ast_{-2}\psi^\ast_{-1}\ket{0}, & m<0,
\end{cases}
\]
and
\[
\bra{m}=
\begin{cases}
\bra{0}\psi^\ast_0\psi^\ast_1\dots \psi^\ast_{m-1}, & m\geq 0,\\
\bra{0}\psi_{-1}\psi_{-2}\dots \psi_{m}, & m<0.
\end{cases}
\]

We define an anti-algebra involution on $\mathcal{A}$ by
\[
{}^\ast:\mathcal{A}\to \mathcal{A};\quad \psi_n\leftrightarrow \psi_n^\ast,
\]
that is, a $k$-linear isomorphism with $(ab)^\ast=b^\ast a^\ast$ and $(a^\ast)^\ast=a$.
Further, we have an isomorphism of $k$-spaces
\[
\omega:\mathcal{F}\to {\mathcal{F}}^\ast,\quad X\ket{0}\mapsto \bra{0}{X}^\ast.
\]

The \textit{vacuum expectation value} is the unique $k$-bilinear map
\[
{\mathcal{F}}^\ast\otimes_k\mathcal{F}\to k,\quad 
\bra{w}\otimes \ket{v}\mapsto \langle{w}\vert v\rangle
\]
that is determined by $\langle 0\vert 0\rangle=1$,
$(\bra{w}\psi_n) \ket{v}=\bra{w} (\psi_n\ket{v})$,
and
$(\bra{w}\psi_n^\ast) \ket{v}=\bra{w} (\psi_n^\ast\ket{v})$.
For any expression $X$, we write 
$\bra{w}X\ket{v}:=(\langle{w}\vert X)\ket{v}
=\bra{w}(\vert X\ket{v})
$.
The expectation value $\bra{0}X\ket{0}$ is often abbreviated as $\langle X\rangle$.

\begin{rem}
The existence of the vacuum expectation value can be checked by using the infinite wedge presentation as follows.
Let $(\cdot,\cdot)$ be the non-degenerate symmetric $k$-bilinear form on $F$ where the set of vectors \eqref{eq:basis} are orthonormal.
From \eqref{eq:fermion_action}, we have $(\psi_i\ket{v},\ket{w})=(\ket{v},\psi_i^\ast\ket{w})$, which means $\psi_i$ and $\psi^\ast_i$ are adjoint to each other.
The vacuum expectation value is defined by $\langle{w}\vert v\rangle:=(\omega(\ket{w}),\ket{v})$.
\end{rem}

\subsection{Wick's theorem}\label{sec:Wick}

In many cases, vacuum expectation values can be calculated by using the following \textit{Wick's theorem}.

\begin{thm}[Wick's theorem
(see {\cite[\S 2]{alexandrov2013free}, \cite[Exercise 4.2]{miwa2012solitons}})
]\label{thm:Wick}
Let $\{m_1,\dots,m_r\}$ and $\{n_1,\dots,n_{r}\}$ be sets of integers.
Then we have
\[
\langle 
\psi_{m_1}\cdots\psi_{m_{r}}
\psi^\ast_{n_r}\cdots\psi^\ast_{n_{1}}
\rangle
=\det(\langle \psi_{m_i}\psi^\ast_{n_j} \rangle)_{1\leq i,j\leq r}.
\]
\end{thm}

For sets of integers $m=\{m_1,\dots,m_r\}$, $n=\{n_1,\dots,n_s\}$ with $m_1>\dots>m_r$, $n_1>\dots>n_s$, we write
\[
\delta_{m,n}=
\begin{cases}
1,& r=s \mbox{ and }m_i=n_i \mbox{ for all $i$},\\
0,&\mbox{otherwise}.
\end{cases}
\]
\begin{cor}\label{cor:dual}
Let $m=\{m_1,\dots,m_r\},n=\{n_1,\dots,n_s\}$ be sets of integers with $m_1>\dots>m_{r}> -r$ and $n_1>\dots>n_{s}> -s$.
Then, 
\[
\bra{-r}
\psi^\ast_{m_r}\cdots\psi^\ast_{m_{1}}
\psi_{n_1}\cdots\psi_{n_{s}}
\ket{-s}
=\delta_{m,n}.
\]
\end{cor}

\subsection{The boson-fermion correspondence}\label{sec:boson-fermion}

Let $:\bullet :$ be the \textit{normal ordering} (see \cite[\S 2]{alexandrov2013free}, \cite[\S 5.2]{miwa2012solitons}) of free-fermions defined as follows: all annihilation operators ($\psi_m$ $(m<0)$ and $\psi^\ast_n$ $(n\geq 0)$) are moved to the right, and an appropriate sign factor $(\pm 1)$ is multiplied. 
For example, we have $:\psi_1\psi^\ast_1:=\psi_1\psi^\ast_1$ and $:\psi^\ast_1\psi_1:=-\psi_1\psi^\ast_1$.

For $m\in \ZZ$, we define an operator $a_m$ as $a_m=\sum_{k\in \ZZ} :\psi_k\psi^\ast_{k+m}:$ on $\mathcal{F}$.
The operator $a_m$ satisfies the following commutative relations
\begin{equation}\label{eq:relation_added}
[a_m,a_n]=m\delta_{m+n,0},\qquad
[a_m,\psi_n]=\psi_{n-m},\qquad
[a_m,\psi^\ast_n]=-\psi^\ast_{n+m},
\end{equation}
where $[A,B]=AB-BA$.
For proofs of these facts, see \cite[\S 5.3]{miwa2012solitons}.
We also note that, if $\ket{v}\leftrightarrow \bra{w}$ under the involution $\omega$, we have $a_n\ket{v}\leftrightarrow\bra{w}a_{-n}$.


Let $x_1,x_2,\dots$ be formal independent variables.
We set
\[
H(x)=\sum_{n>0}\frac{p_n(x)}{n}a_n,\qquad p_n(x)=x_1^n+x_2^n+\cdots\mbox{,\ (the $n$-th power sum)},
\]
which satisfies the commutative relation
\begin{equation}\label{eq:exp_H_psi}
e^{H(x)}\psi_n=\left(\sum_{i=0}^\infty h_i(x)\psi_{n-i}\right)e^{H(x)},
\end{equation}
where $h_i(x)$ ($i\in \ZZ_{\geq 0}$) is the $i$-th complete symmetric function.

\begin{thm}[{\cite[Lemma 9.5]{miwa2012solitons}} (also see {\cite[Theorem 6.1]{kac2013bombay}})]\label{thm:Schur}
Let $\lambda=(\lambda_1\geq \dots\geq \lambda_\ell> 0)$ be a partition of length $\ell$.
We set $\lambda_{\ell+1}=\lambda_{\ell+2}=\dots=\lambda_{r}=0$ for $r\geq \ell$.
Then we have
\[
s_\lambda(x)=
\bra{0}e^{H(x)}
\psi_{\lambda_1-1}\psi_{\lambda_2-2}\cdots\psi_{\lambda_r-r}\ket{-r}
=
\det(h_{\lambda_i-i+j}(x))_{1\leq i,j\leq r}
,
\]
where $s_\lambda(x)$ is the Schur function.
\end{thm}

\subsection{Operators $e^\Theta$ and $e^\theta$}\label{sec:etheta_eTheta}

Let $\psi(z)$ and $\psi^\ast(z)$ be the generating functions of $\psi_n$ and $\psi^\ast_n$ with a formal variable $z$:
\[
\psi(z)=\sum_{n\in \ZZ}\psi_nz^n,\qquad
\psi^\ast(z)=\sum_{n\in \ZZ}\psi^\ast_nz^n.
\]
Note that $[a_n,\psi(z)]=z^n\psi(z)$ and $[a_n,\psi^\ast(z)]=-z^{-n}\psi^\ast(z)$.

We now introduce two important operators
\[
\Theta=\beta a_{-1}-\frac{\beta^2}{2}a_{-2}+\frac{\beta^3}{3}a_{-3}-\cdots
\ \mbox{ and }\ 
\theta=\beta a_{1}-\frac{\beta^2}{2}a_{2}+\frac{\beta^3}{3}a_{3}-\cdots.
\]
Since
\[
e^X Ye^{-X}=
Y+\mathrm{ad}_X\cdot Y+\frac{(\mathrm{ad}_X)^2}{2{\color{red}!}}\cdot Y+\cdots,
\]
($\mathrm{ad}_X\cdot Y=[X,Y]$), we have the following equations
\begin{equation}\label{eq:e_theta_action}
e^{\Theta}\psi(z)e^{-\Theta}=(1+\beta z^{-1})\psi(z),\qquad
e^{\theta}\psi(z)e^{-\theta}=(1+\beta z)\psi(z).
\end{equation}

We give a list of lemmas that are useful in the following sections.

\begin{lemma}\label{lemma:e_theta_action_elem}
$e^\Theta \psi_ne^{-\Theta}=
\psi_n+\beta \psi_{n+1}
$,
$e^\theta \psi_ne^{-\theta}=
\psi_n+\beta \psi_{n-1}
$.
\end{lemma}

\begin{lemma}
$
e^{H(x)}e^{\Theta}=
\prod_{i=1}^\infty
(1+\beta x_i)
\cdot
e^{\Theta}e^{H(x)}
$.
\end{lemma}

\begin{lemma}\label{lemma:e_theta_action_elem_end}
$
e^{H(x)}\psi(z)=
\left(
\sum_{i=0}^{\infty}h_i(x)z^i
\right)
\psi(z)e^{H(x)}
$.
\end{lemma}
Lemmas \ref{lemma:e_theta_action_elem}--\ref{lemma:e_theta_action_elem_end} can be shown from  (\ref{eq:exp_H_psi}--\ref{eq:e_theta_action}) by straightforward calculations.

\begin{lemma}\label{lemma:new_lemma_added}
For $m\in \ZZ$ and $s\geq 1$, we have
\[
\psi_{m-1}e^{\Theta}
\psi_{m-2}e^{\Theta}
\cdots
\psi_{m-s}e^{\Theta}
=
\psi_{m-1}
\psi_{m-2}
\cdots
\psi_{m-s}e^{s\Theta}.
\]
\end{lemma}
\begin{proof}
We prove the lemma by induction on $s\geq 1$.
If $s=1$, the equation is trivial.
For $s\geq 1$, we have by induction hypothesis
\begin{align*}
&\psi_{m-1}e^{\Theta}
\psi_{m-2}e^{\Theta}
\cdots
\psi_{m-s}e^{\Theta}
\psi_{m-s-1}e^{\Theta}\\
&=\psi_{m-1}e^{\Theta}(\psi_{m-2}\psi_{m-3}\dots \psi_{m-s-1}e^{s\Theta})\\
&=\psi_{m-1}(\psi_{m-2}+\beta \psi_{m-1})(\psi_{m-3}+\beta \psi_{m-2})\dots (\psi_{m-s-1}+\beta \psi_{m-s})e^{(s+1)\Theta}.
\end{align*}
Since $\psi_{m-1}^2=\psi_{m-2}^2=\dots=\psi_{m-s}^2=0$, the last expression can be rewritten as
\[
\psi_{m-1}\psi_{m-2}\psi_{m-3}\dots \psi_{m-s-1}e^{(s+1)\Theta}.
\]
\end{proof}

\begin{lemma}
For $m\in \ZZ$ and $s\geq 1$, we have
\[
\psi_{m-1}e^{\Theta}
\psi_{m-2}e^{\Theta}
\cdots
\psi_{m-s}e^{\Theta}
\psi_{m-s}e^{\Theta}
=
(-\beta)^{s}
\psi_{m}e^{\Theta}
\psi_{m-1}e^{\Theta}
\cdots
\psi_{m-s+1}e^{\Theta}
\psi_{m-s}e^{\Theta}.
\]
\end{lemma}
\begin{proof}
We prove the lemma by induction on $s$.
When $s=1$, it follows that
\begin{align*}
\psi_{m-1}e^{\Theta}
\psi_{m-1}e^{\Theta}
&=
e^{\Theta}(\psi_{m-1}-\beta \psi_{m}+\beta^2\psi_{m+1}-\cdots)
\psi_{m-1}e^{\Theta}\\
&=
e^{\Theta}(-\beta \psi_{m}+\beta^2\psi_{m+1}-\cdots)
\psi_{m-1}e^{\Theta}\\
&=(-\beta)
\psi_{m}e^{\Theta}
\psi_{m-1}e^{\Theta}
\end{align*}
from $\psi_{m-1}^2=0$.
For general $s$, we have
\begin{align*}
&
\psi_{m-1}e^{\Theta}
\cdots
\psi_{m-s-1}e^{\Theta}
\psi_{m-s}e^{\Theta}
\psi_{m-s}e^{\Theta}\\
&=
(-\beta)
\psi_{m-1}e^{\Theta}
\cdots
\psi_{m-s-1}e^{\Theta}
\psi_{m-s-1}e^{\Theta}
\psi_{m-s}e^{\Theta}\\
&=
(-\beta)^s
\psi_{m}e^{\Theta}
\cdots
\psi_{m-s-2}e^{\Theta}
\psi_{m-s-1}e^{\Theta}
\psi_{m-s}e^{\Theta}
\end{align*}
by induction hypothesis.
\end{proof}
\begin{cor}\label{cor:X(n)}
Set 
$
X(n):=
\psi_{n_1-1}e^{\Theta}
\psi_{n_2-2}e^{\Theta}
\cdots
\psi_{n_r-r}e^{\Theta}
$
for $n=(n_1,\dots,n_r)$.
Assume $n_1-1\geq n_2-2\geq \dots\geq n_r-r$.
Then the following equation holds:
\[
X(n)=(-\beta)^{\zet{\overline{n}}-\zet{n}}\cdot 
X(\overline{n}),
\]
where $\overline{n}_j=\max[n_j,n_{j+1},\dots,n_r]$, $\zet{n}=\sum_{j=1}^r n_j$, and $\zet{\overline{n}}=\sum_{j=1}^r \overline{n}_j$.
\end{cor}

Similarly, we have:
\begin{lemma}
For $m\in \ZZ$, it follows that
\[
\psi_me^{-\theta}
\psi_m
=(-\beta)
\psi_me^{-\theta}
\psi_{m-1}.
\]
\end{lemma}
\begin{proof}
Since $\psi_m^2=0$, we have
$
\psi_me^{-\theta}\psi_m
=
\psi_m(\psi_m-\beta \psi_{m-1}+\beta^2\psi_{m+2}-\cdots)e^{-\theta}
=
\psi_m(-\beta \psi_{m-1}+\beta^2\psi_{m+2}-\cdots)e^{-\theta}
=
(-\beta)\psi_me^{-\theta}\psi_{m-1}
$.
\end{proof}
\begin{cor}\label{cor:x(n)}
Set 
$
x(n):=
\psi_{n_1-1}e^{-\theta}
\psi_{n_2-2}e^{-\theta}
\cdots
\psi_{n_r-r}e^{-\theta}
$ for $n=(n_1,\dots,n_r)$.
Assume $n_1-1\geq n_2-2\geq \dots\geq n_r-r$.
Then the following equation holds:
\[
x(n)=(-\beta)^{\zet{n}-\zet{\underline{n}}}\cdot 
x(\underline{n}),
\]
where $\underline{n}_j=\min[n_1,\dots,n_j]$.
\end{cor}

\section{Stable Grothendieck polynomial $G_\lambda(x)$}\label{sec:Gr}

\subsection{Definition of $G^r_\lambda(x)$}\label{sec:def_of_Gr}

Let $\lambda=(\lambda_1\geq \dots\geq \lambda_\ell>0)$ be a partition.
For $r\geq \ell$, we put $\lambda_{\ell+1}=\lambda_{\ell+2}=\dots=\lambda_{r}=0$.
Let $G^r_\lambda(x)$ denote the symmetric function that is defined by
\[
G^r_\lambda(x):=
\bra{0}
e^{H(x)}
\psi_{\lambda_1-1}e^{\Theta}
\psi_{\lambda_2-2}e^{\Theta}
\cdots
\psi_{\lambda_r-r}e^{\Theta}
\cdot e^{-r\Theta}
\ket{-r}.
\]
This expression is rewritten as
\[
G^r_\lambda(x)=
\bra{0}
e^{H(x)}
\psi_{\lambda_1-1}e^{\Theta}
\psi_{\lambda_2-2}e^{\Theta}
\cdots
\psi_{\lambda_\ell-\ell}e^{\Theta}
\cdot
\psi_{-\ell-1}
\cdots
\psi_{-r}
\cdot e^{-\ell\Theta}
\ket{-r}
\]
by using Lemma \ref{lemma:new_lemma_added}.

\begin{prop}\label{prop:G_and_Gr}
If $\ell(\lambda)\leq n\leq r$, we have
\[
G_\lambda(x_1,\dots,x_n)=
G^r_\lambda(x_1,\dots,x_n,0,0,\dots).
\]
\end{prop}
\begin{proof}
Let us consider the generating function
\begin{equation}\label{eq:genrating_function_Gr}
\Psi(z_1,\dots,z_r):=
\bra{0}
e^{H(x)}
\psi(z_1)e^{\Theta}
\psi(z_2)e^{\Theta}
\cdots
\psi(z_r)e^{\Theta}
\cdot e^{-r\Theta}
\ket{-r}
\end{equation}
of $G^r_\lambda(x)$.
Set $A_i=e^{(i-1)\Theta}\psi(z_i)e^{-(i-1)\Theta}=(1+\beta z_i^{-1})^{i-1}\psi(z_i)$.
From Wick's theorem (Theorem \ref{thm:Wick}), it follows that 
\begin{align*}
\Psi(z_1,\dots,z_r)
&=
\bra{0}
e^{H(x)}
A_1A_2\dots A_r
\ket{-r}
=
\det
(
\bra{0}
e^{H(x)}
A_i
e^{-H(x)}
\psi^\ast_{-j}
\ket{0}
)_{1\leq i,j\leq r}.
\end{align*}
By substituting $e^{H(x)}A_ie^{-H(x)}=(1+\beta z_i^{-1})^{i-1}(\textstyle \sum_{m=0}^{\infty}h_m(x)z_i^m)\psi(z_i)$, which follows from Lemma \ref{lemma:e_theta_action_elem_end}, we have
\begin{align*}
\Psi(z_1,\dots,z_r)
&=
\det
\left(
(1+\beta z_i^{-1})^{i-1}
(\textstyle \sum_{m=0}^{\infty}h_m(x)z_i^m)
\bra{0}
\psi(z_i)
\psi^\ast_{-j}
\ket{0}
\right)_{1\leq i,j\leq r}\\
&=
\det
\left(
(1+\beta z_i^{-1})^{i-1}
(\textstyle \sum_{m=0}^{\infty}h_m(x)z_i^m)
z_i^{-j}
\right)_{1\leq i,j\leq r}.
\end{align*}
Comparing the coefficients of $z_1^{\lambda_1-1}z_2^{\lambda_2-2}\cdots z_r^{\lambda_r-r}$ on the both sides, we obtain
\[
G^r_\lambda(x)=\det
\left(
\sum_{m=0}^{\infty}
\binom{i-1}{m}
\beta^{m}
h_{\lambda_i-i+j+m}(x)
\right)_{1\leq i,j\leq r}.
\]
From (\ref{eq:Jacobi-Trudi}), we have the desired result.
\end{proof}


\subsection{The completed ring $\widehat{\Lambda}$}\label{sec:completion}

Put $k=\CC(\beta)$.
Let $\Lambda$ be the $k$-algebra of symmetric functions \cite[\S I.2]{macdonald1998symmetric} in $x_1,x_2,\dots$.
In this section we give a brief review on the completed ring $\widehat{\Lambda}\supset \Lambda$.

Let $M_n$ ($n\geq 1$) be the $k$-subspace of $\Lambda$ that is expressed as 
\[
M_n:=\left\{\sum_{i=1}^N c_{\lambda_i} s_{\lambda_i}(x)\
;
\, N\geq 0,
\, \lambda_1,\dots,\lambda_N \mbox{ are partitions},
\, c_{\lambda_i}\in k,
\, \ell(\lambda_i)\geq n,
 \right\},
\]
where $\ell(\lambda)$ is the length of a partition $\lambda$.
Since $M_n\supset M_{n+1}$, an inverse system
\[
\Lambda/M_1\leftarrow \Lambda/M_2\leftarrow\Lambda/M_3 \leftarrow \cdots
\]
of $k$-spaces exists.
Let $\widehat{\Lambda}:=\lim\limits_{\longleftarrow} (\Lambda/M_n)$ be the inverse limit.
Note that there exists a natural inclusion $\Lambda\hookrightarrow \widehat{\Lambda}$.

It is convenient to introduce a $k$-linear topology on $\Lambda$ where the family $\{M_n\}_{n=1,2,\dots}$ forms an open neighborhood base at $0$.
In terms of this topology, the inclusion $\Lambda \hookrightarrow\widehat{\Lambda}$ can be viewed as a completion of the topological space $\Lambda$.
Note that
\[
f(x)\in M_{n+1} \iff f(x_1,\dots,x_{n},0,0,\dots)=0.
\]

Moreover, $\widehat{\Lambda}$ is indeed a topological $k$-algebra, over which the multiplication is also continuous.
\begin{lemma}\label{lemma:small_polynomial_1}
If $n_1,\dots,n_r>-r$, then
$
\bra{0}
e^{H(x)}\psi_{n_1}\dots \psi_{n_r}
\ket{-r}\in M_r
$.
\end{lemma}
\begin{proof}
It follows from Theorem \ref{thm:Schur}.
\end{proof}

It is known that there exists a unique element $G_\lambda(x)\in \widehat{\Lambda}$, which is called the \textit{stable Grothendieck polynomial}~\cite{fomin1994grothendieck}, that satisfies the equation
\[
G_\lambda(x_1,\dots,x_n)=G_\lambda(x_1,\dots,x_n,0,0,\dots)
\]
for any $n$.
From Proposition \ref{prop:G_and_Gr}, we have
\begin{equation}\label{eq:GandGr}
G_\lambda(x)-G^r_\lambda(x)\in M_{n+1}\qquad
\mbox{for any }\ell(\lambda)\leq n\leq r,
\end{equation}
which implies the fact that `$G_\lambda(x)$ and $G^r_\lambda(x)$ are sufficiently near.' 
From \eqref{eq:GandGr}, by putting $n=r$, we have $G_\lambda(x)-G^r_\lambda(x)\in M_{r+1}$.
As a subset of the topological space $\widehat{\Lambda}$, the sequence $G^1_\lambda(x),G^2_\lambda(x),\dots\in \widehat{\Lambda}$ converges to $G_\lambda(x)$. 
We simply write this fact as
\begin{equation}\label{eq:limG^r}
G_\lambda(x)=\lim\limits_{r\to \infty}G^r_\lambda(x).
\end{equation}

\subsection{Remarks on elements of $\widehat{\Lambda}$}

We will often interested in symmetric functions of the form
\begin{equation}\label{eq:new_eq_considering}
\bra{0}e^{H(x)}\psi_{m_1}\dots \psi_{m_r}e^{s\Theta}\ket{-r}
\end{equation}
where $r,s\geq 0$ and $m_1,\dots,m_r>-r$.
In general, such symmetric function cannot be contained in $\Lambda$.
If fact, if $r=0$ and $s=1$, we have
\[
\bra{0}e^{H(x)}e^{\Theta}\ket{0}=
\prod_{i=1}^{\infty}(1+\beta x_i)\cdot
\bra{0}e^{\Theta}e^{H(x)}\ket{0}
=\prod_{i=1}^{\infty}(1+\beta x_i)
\in \widehat{\Lambda}\setminus\Lambda.
\]
We can check that, if we substitute $x_{n+1}=x_{n+2}=\cdots=0$, this function reduces to a symmetric polynomial in $n$ variables.
The following lemma states that the same is true for any symmetric function of the form \eqref{eq:new_eq_considering}.
\begin{lemma}\label{lemma:f(x1,,,xn)}
Let 
$
H(x_1,\dots,x_n):=H(x)\vert_{x_{n+1}=x_{n+2}=\cdots=0}.
$
Then
\begin{equation}\label{eq:f(x)}
f(x_1,\dots,x_n):=
\bra{0}e^{H(x_1,\dots,x_n)}\psi_{m_1}\dots \psi_{m_r}e^{s\Theta}\ket{-r}
\end{equation}
is a symmetric polynomial in $x_1,\dots,x_n$.
\end{lemma}
\begin{proof}
Let
$
X_{-r}:=
\binom{s}{1}\beta\psi_{-r}+\binom{s}{2}\beta^2\psi_{-r+1}+\cdots
+\binom{s}{s}\beta^s\psi_{-r-1+s}
$.
Since $e^{s\Theta}\psi_{-r-1}=(\psi_{-r-1}+X_{-r})e^{s\Theta}$, we have
\begin{align}
&\bra{0}e^{H(x_1,\dots,x_n)}\psi_{m_1}\dots \psi_{m_r}e^{s\Theta}\ket{-r}\label{eq:infiniteseries}\\
&=
\bra{0}e^{H(x_1,\dots,x_n)}\psi_{m_1}\dots \psi_{m_r}\ket{-r}\nonumber\\
&\hspace{20pt}+
\bra{0}e^{H(x_1,\dots,x_n)}\psi_{m_1}\dots \psi_{m_r}X_{-r}\ket{-r-1}\nonumber\\
&\hspace{20pt}+
\bra{0}e^{H(x_1,\dots,x_n)}\psi_{m_1}\dots \psi_{m_r}(\psi_{-r-1}+X_{-r})X_{-r-1}\ket{-r-2}\nonumber\\
&\hspace{20pt}+
\bra{0}e^{H(x_1,\dots,x_n)}\psi_{m_1}\dots \psi_{m_r}(\psi_{-r-1}+X_{-r})(\psi_{-r-2}+X_{-r-1})X_{-r-2}\ket{-r-3}\nonumber\\
&\hspace{20pt}+\cdots.\nonumber
\end{align}
Because
\[
\bra{0}e^{H(x_1,\dots,x_n)}\psi_{m_1}\dots \psi_{m_r}(\psi_{-r-1}+X_{-r})\cdots(\psi_{-r-t}+X_{-r-t+1})X_{-r-t}\ket{-r-t-1}
\]
is $0$ if $r+t>n$ (see Lemma \ref{lemma:small_polynomial_1}), the right hand side on (\ref{eq:infiniteseries}) is in fact a polynomial.
\end{proof}

From Lemma \ref{lemma:f(x1,,,xn)}, $f(x_1,\dots,x_n)$ in (\ref{eq:f(x)}) determines a unique element of $\Lambda/M_{n+1}$.
Because $f(x_1,\dots,x_n)=f(x_1,\dots,x_n,0)$ for any $n$, there uniquely exists an element $f(x)=f(x_1,x_2,\dots)\in \widehat{\Lambda}$ which satisfies $f(x_1,\dots,x_n)=f(x_1,\dots,x_n,0,0,\dots)$.
In other words, we have
\[
f(x)=\bra{0}e^{H(x)}\psi_{m_1}\dots \psi_{m_r}e^{s\Theta}\ket{-r}.
\]

The expression (\ref{eq:infiniteseries}) implies that:
\begin{prop}\label{prop:near_criterion}
$
\bra{0}e^{H(x)}\psi_{m_1}\dots \psi_{m_r}(e^{s\Theta}-1)\ket{-r}\in M_{r+1}
$.
\end{prop}

The following lemma will be useful in the next section.
\begin{lemma}\label{lemma:various_expressions}
Let $m_1,\dots,m_\ell$ be a sequence of integers.
Then
\begin{align*}
&\bra{0}e^{H(x)}
\psi_{m_1}\dots \psi_{m_\ell}
\psi_{-\ell-1}e^{\Theta}
\psi_{-\ell-2}e^{\Theta}
\cdots
\psi_{-r}e^{\Theta}
\ket{-r}\\
&\hspace{10pt}=
\bra{0}e^{H(x)}
\psi_{m_1}\dots \psi_{m_\ell}
\psi_{-\ell-1}
\psi_{-\ell-2}
\cdots
\psi_{-r}e^{(r-\ell)\Theta}
\ket{-r}\\
&\hspace{10pt}=
\bra{0}e^{H(x)}
\psi_{m_1}\dots \psi_{m_\ell}
\psi_{-\ell-1}
\psi_{-\ell-2}
\cdots
\psi_{-r}
\ket{-r}\\
&\hspace{10pt}=
\bra{0}e^{H(x)}
\psi_{m_1}\dots \psi_{m_\ell}
\ket{-\ell}.
\end{align*}
\end{lemma}
\begin{proof}
The first equality follows from Lemma \ref{lemma:new_lemma_added}.
Let
\[\textstyle
Y_{-r}:=
\binom{r-\ell}{1}\beta\psi_{-r}+\binom{r-\ell}{2}\beta^2\psi_{-r+1}+\cdots
+\binom{r-\ell}{r-\ell}\beta^{r-\ell}\psi_{-\ell-1}.
\]
Because $\psi_{-\ell-1}
\psi_{-\ell-2}
\cdots
\psi_{-r-t}Y_{-r-t}=0$ for any $t\geq 0$, the equation (\ref{eq:infiniteseries}) is now simplified as
\begin{align*}
&\bra{0}e^{H(x_1,\dots,x_n)}
\psi_{m_1}\dots \psi_{m_\ell}
\psi_{-\ell-1}
\psi_{-\ell-2}
\cdots
\psi_{-r}e^{(r-\ell)\Theta}
\ket{-r}\\
&\hspace{20pt}=
\bra{0}e^{H(x_1,\dots,x_n)}
\psi_{m_1}\dots \psi_{m_\ell}
\psi_{-\ell-1}
\psi_{-\ell-2}
\cdots
\psi_{-r}
\ket{-r},
\end{align*}
which implies the second equality.
The third equality is obvious.
\end{proof}

\begin{rem}
If $s'<0$, the expression $\bra{0}e^{H(x)}\psi_{m_1}\dots \psi_{m_r}e^{s'\Theta}\ket{-r}$ does not determine an element of $\widehat{\Lambda}$.
In fact, if $r=0$ and $s'=-1$, the expression is rewritten as
$\bra{0}e^{H(x)}e^{-\Theta}\ket{0}=\prod_{i=1}^\infty (1+\beta x_i)^{-1}$, which is not contained in $\widehat{\Lambda}$.
\end{rem}

\subsection{Free-fermionic expression of $G_\lambda(x)$}
Let us consider the symmetric function
\[
\oG_\lambda(x):=
\bra{0}
e^{H(x)}
\psi_{\lambda_1-1}e^{\Theta}
\psi_{\lambda_2-2}e^{\Theta}
\cdots
\psi_{\lambda_\ell-\ell}e^{\Theta}
\ket{-\ell}.
\]
Note that the symmetric functions $G^r_\lambda(x)$ and $\oG_\lambda(x)$ are quite similar but different.
Assume $r\geq \ell$.
From Lemma \ref{lemma:various_expressions}, their difference is expressed as
\begin{align*}
&\oG_\lambda(x)-G^r_\lambda(x)\\
&=
\bra{0}
e^{H(x)}
\psi_{\lambda_1-1}e^{\Theta}
\cdots
\psi_{\lambda_\ell-\ell}e^{\Theta}
\psi_{-\ell-1}
\cdots
\psi_{-r}
e^{-\ell\Theta}
(e^{\ell\Theta}-1)
\ket{-r}.
\end{align*}
From this equation and Proposition \ref{prop:near_criterion}, we have
\begin{equation}\label{eq:oGrandGr}
\oG_\lambda(x)-G^r_\lambda(x)\in M_{r+1},
\end{equation}
which implies that the sequence $\{G^r_\lambda(x)\}_{r=1,2,\dots}$ converges to $\oG_\lambda(x)$ in $\widehat{\Lambda}$.
It follows from \eqref{eq:limG^r} that
\[
\oG_\lambda(x)
=
\lim\limits_{r\to \infty}G^r_\lambda(x)
=
G_\lambda(x).
\]
In other words, we have:
\begin{thm}\label{thm:main}
\[
G_\lambda(x)=\bra{0}
e^{H(x)}
\psi_{\lambda_1-1}e^{\Theta}
\psi_{\lambda_2-2}e^{\Theta}
\cdots
\psi_{\lambda_\ell-\ell}e^{\Theta}
\ket{-\ell}.
\]
\end{thm}

\subsection{``Another'' determinant formula for $G_\lambda(x)$}\label{sec:another_determinantial_formula}

We often write $G_n(x)=G_{(n)}(x)$, where $(n)$ is a partition of length $1$.
\begin{prop}\label{prop:G_n}
We have
\[
\sum_{n\in \ZZ} G_n(x)z^n=\frac{1}{1+\beta z^{-1}}\prod_{i=1}^\infty \frac{1+\beta x_i}{1-x_iz},
\]
where $G_n(x)=\bra{0}e^{H(x)}\psi_{n-1}e^{\Theta}\ket{-1}$ and 
$(1+\beta z^{-1})^{-1}=\sum_{n=0}^{\infty}(-\beta)^nz^{-n}$.
\end{prop}
\begin{proof}
\begin{align*}
\sum_{n\in \ZZ} G_n(x)z^n
&=
\bra{0}e^{H(x)}\psi(z)ze^{\Theta}\ket{-1}
=
(1+\beta z^{-1})^{-1}
\bra{0}e^{H(x)}e^{\Theta}\psi(z)z\ket{-1}\\
&=
(1+\beta z^{-1})^{-1}
\textstyle
\prod_{i=1}^\infty(1+\beta x_i)
\bra{0}e^{H(x)}\psi(z)z\ket{-1}\\
&=
(1+\beta z^{-1})^{-1}
\textstyle
\prod_{i=1}^\infty(1+\beta x_i)
\cdot (\sum_{i=0}^\infty h_i(x)z^i)\\
&=
\frac{1}{1+\beta z^{-1}}
\prod_{i=1}^\infty\frac{1+\beta x_i}{1-x_iz}.
\end{align*}
\end{proof}
Let $\mathcal{G}(z):=\sum_{n\in \ZZ} G_n(x)z^n$.
From the proof of Proposition \ref{prop:G_n}, we derive the commutative relation
\begin{equation}\label{eq:comG}
\textstyle
e^{H(x)}e^{-\Theta}\psi(z)e^{\Theta}e^{-H(x)}=
\prod_{i=1}^\infty(1+\beta x_i)^{\color{red}-1}\cdot \mathcal{G}(z)\psi(z).
\end{equation}

We consider the formal function
\begin{equation}\label{eq:formalPsi}
\overline{\Psi}(z_1,\dots,z_r):=
\bra{0}e^{H(x)}\psi(z_1)e^{\Theta}\cdots \psi(z_r)e^{\Theta}\ket{-r},
\end{equation}
which is a generating function of $G_\lambda(x)$.
\begin{prop}[\cite{HUDSON2017115}, see also \cite{nakagawa2018universalfactrial}]
\label{prop:another_determinant_G}
We have
\[
G_\lambda(x)=\det
\left(
\sum_{m=0}^{\infty}
\binom{i-r}{m}\beta^m
G_{\lambda_i-i+j+m}(x)
\right)_{1\leq i,j\leq r}.
\]
\end{prop}
\begin{proof}
Let $B_i:=e^{-(r-i+1)\Theta}\psi(z_i)e^{(r-i+1)\Theta}$.
Applying Wick's theorem (Theorem \ref{thm:Wick}) to the generating function (\ref{eq:formalPsi}) gives
\begin{align*}
\overline{\Psi}(z_1,\dots,z_r)
&=
\bra{0}e^{H(x)}e^{r\Theta}B_1B_2\dots B_r\ket{r}\\
&=
\textstyle\prod_{l=1}^\infty(1+\beta x_l)^{\color{red}r}
\bra{0}e^{H(x)}B_1B_2\dots B_r\ket{-r}\\
&=
\textstyle\prod_{l=1}^\infty(1+\beta x_l)^{\color{red}r}
\cdot
\det
(
\bra{0}
e^{H(x)}
B_i
e^{-H(x)}
\psi^\ast_{-j}
\ket{0}
)_{1\leq i,j\leq r}.
\end{align*}
Since
\begin{align*}
e^{H(x)}B_ie^{-H(x)}
&=
(1+\beta z_i^{-1})^{-(r-i)}
e^{H(x)}e^{-\Theta}
\psi(z_i)
e^{\Theta}e^{-H(x)}\\
&=
\textstyle
\prod_{l=1}^\infty(1+\beta x_l{\color{red})}^{-1}\cdot 
(1+\beta z_i^{-1})^{-(r-i)}
\mathcal{G}(z_i)
\psi(z_i)
\end{align*}
(see (\ref{eq:comG})), we have
\begin{align*}
\overline{\Psi}(z_1,\dots,z_r)
&=
\det
\left((1+\beta z_i^{-1})^{-(r-i)}
\mathcal{G}(z_i)
\bra{0}
\psi(z_i)
\psi^\ast_{-j}
\ket{0}
\right)_{1\leq i,j\leq r}\\
&=
\det
\left((1+\beta z_i^{-1})^{-(r-i)}
\mathcal{G}(z_i)z_i^{-j}
\right)_{1\leq i,j\leq r}.
\end{align*}
Comparing the coefficients of $z_1^{\lambda_1-1}\cdots z_r^{\lambda_r-r}$ on the both sides, we obtain the desired equation. 
\end{proof}

\section{Dual stable Grothendieck polynomial $g_\lambda(x)$}\label{sec:gr}

\subsection{Definition}\label{sec:def_of_g}

For a partition $\lambda=(\lambda_1\geq \lambda_2\geq \dots\geq \lambda_\ell> 0)$ and $\lambda_{\ell+1}=\dots=\lambda_r=0$, we set
\[
g_\lambda(x):=
\bra{0}
e^{H(x)}
\psi_{\lambda_1-1}e^{-\theta}
\psi_{\lambda_2-2}e^{-\theta}
\cdots
\psi_{\lambda_r-r}e^{-\theta}
\ket{-r}.
\]
Note the the definition of $g_\lambda(x)$ does not depend on the choice of $r\geq \ell$ because of the equation
$
\ket{-r}=
\psi_{-r-1}\ket{-r-1}=
\psi_{-r-1}e^{-\theta}\ket{-r-1}
$.

\subsection{Proof of the duality}

The \textit{Hall inner product} 
$\left\langle\cdot ,\cdot \right\rangle:\Lambda\times \Lambda\to k$ is the non-degenerate $k$-bilinear form that satisfies $\left\langle s_\lambda(x),s_\mu(x)\right\rangle=\delta_{\lambda,\mu}$.
The bilinear form can be uniquely extended to the bilinear form $\widehat{\Lambda}\times \Lambda\to k$ continuously.

Let $X=\psi_{n_1}\dots\psi_{n_r}$ and $Y=\psi_{m_1}\dots\psi_{m_s}$.
If two symmetric functions $f(x)$ and $g(x)$ are expressed as $f(x)=\bra{0}e^{H(x)}X\ket{-r}$ and $g(x)=\bra{0}e^{H(x)}Y\ket{-s}$, their Hall inner product can be calculated by using the formula
\[
\left\langle f(x),g(x)\right\rangle
=
\bra{-r}X^\ast Y\ket{-s},
\]
which is obtained from Corollary \ref{cor:dual}.

\begin{prop}\label{prop:duality}
We have
\[
\left\langle
G_\lambda(x),g_\mu(x)
\right\rangle
=\delta_{\lambda,\mu}.
\]
This means that $g_\lambda(x)$ is nothing but the dual stable Grothendieck polynomial.
\end{prop}

To prove Proposition \ref{prop:duality}, it suffices to show
\begin{equation}
\bra{-r}
e^{\theta}\psi^\ast_{\lambda_{r}-r}
\cdots
e^{\theta}\psi^\ast_{\lambda_{2}-2}
e^{\theta}\psi^\ast_{\lambda_{1}-1}
\psi_{\mu_1-1}e^{-\theta}
\psi_{\mu_2-2}e^{-\theta}
\cdots
\psi_{\mu_s-s}e^{-\theta}
\ket{-s}=\delta_{\lambda,\mu}.
\end{equation}
For this, we need the following two lemmas.
\begin{lemma}\label{lemma:a}
If $N>n_1,\dots,n_s$, then
$
\psi^\ast_N
\psi_{n_1}e^{-\theta}
\psi_{n_2}e^{-\theta}
\cdots 
\psi_{n_s}e^{-\theta}
\ket{-s}=0
$.
\end{lemma}
\begin{proof}
As $e^{-\theta}\psi_ne^{\theta}=\psi_n-\beta \psi_{n-1}+\beta^2 \psi_{n-2}-\cdots$, the vector 
$\psi_{n_1}e^{-\theta}
\cdots 
\psi_{n_s}e^{-\theta}
\ket{-s}$
must be a linear combination of vectors of the form 
\[
\psi_{n'_1}\cdots \psi_{n'_s}\ket{-s},\qquad N>n'_1,\dots,n'_s.
\]
Since $[\psi^\ast_m,\psi_n]_+=0$ for $m\neq n$, we obtain the desired result.
\end{proof}
\begin{lemma}\label{lemma:b}
If $M>m_1>\dots>m_r\geq -s$, then
$
\bra{-s}
e^{\theta}\psi^\ast_{m_r}
\dots 
e^{\theta}\psi^\ast_{m_1}
\psi_M=0
$.
\end{lemma}
\begin{proof}
We prove by induction on $r\geq 0$.
If $r=0$ and $M\geq -s$, the equation $\bra{-s}\psi_M=0$ is obvious.
Next assume $r\geq 1$.
Since $e^{\theta}\psi^\ast_{m_1}\psi_M=-(\psi_M+\beta \psi_{M-1})e^{\theta}\psi^\ast_{m_1}$, we have
\begin{align*}
\bra{-s}
e^{\theta}\psi^\ast_{m_r}
\cdots 
e^{\theta}\psi^\ast_{m_2}
e^{\theta}\psi^\ast_{m_1}
\psi_M
=
-\bra{-s}
e^{\theta}\psi^\ast_{m_r}
\cdots 
e^{\theta}\psi^\ast_{m_2}
(\psi_M+\beta \psi_{M-1})e^{\theta}\psi^\ast_{m_1}.
\end{align*}
Because $M-1>m_2$, this equals to $0$ by induction hypothesis.
\end{proof}

\begin{proof}[Proof of Proposition \ref{prop:duality}]
Let
\[
C:=
\bra{-r}
e^{\theta}\psi^\ast_{\lambda_{r}-r}
\cdots
e^{\theta}\psi^\ast_{\lambda_{2}-2}
e^{\theta}\psi^\ast_{\lambda_{1}-1}
\psi_{\mu_1-1}e^{-\theta}
\psi_{\mu_2-2}e^{-\theta}
\cdots
\psi_{\mu_s-s}e^{-\theta}
\ket{-s}.
\]
If $\lambda_1>\mu_1$, then $C=0$ from Lemma \ref{lemma:a}.
If $\lambda_1<\mu_1$, then $C=0$ from Lemma \ref{lemma:b}.
Assume $\lambda_1=\mu_1$.
Since 
$
\psi^\ast_{\lambda_{1}-1}
\psi_{\mu_1-1}
=1-
\psi_{\mu_1-1}
\psi^\ast_{\lambda_{1}-1}
$, $C$ is rewritten as
\[
C=\bra{-r}
e^{\theta}\psi^\ast_{\lambda_{r}-r}
\cdots
e^{\theta}\psi^\ast_{\lambda_{2}-2}
\psi_{\mu_2-2}e^{-\theta}
\cdots
\psi_{\mu_s-s}e^{-\theta}
\ket{-s}
\]
by using Lemma \ref{lemma:a} again.
Repeating this procedure, we conclude that $C=\delta_{\lambda,\mu}$.
\end{proof}


\subsection{Determinant formula for $g_\lambda(x)$}\label{sec:determinant_g}

\begin{prop}[\cite{lascoux2014finite,shimozono2011}]
\label{prop:determinant_g}
We have
\[
g_\lambda(x)=\det
\left(
\sum_{m=0}^{\infty}\binom{1-i}{m}\beta^m h_{\lambda_i-i+j-m}(x)
\right)_{1\leq i,j\leq r}.
\]
\end{prop}
\begin{proof}
Let
\[
\Phi(z_1,\dots,z_r)=
\bra{0}e^{H(x)}
\psi(z_1)e^{-\theta}
\psi(z_2)e^{-\theta}
\cdots
\psi(z_r)e^{-\theta}
\ket{-r}.
\]
We put $D_i:=e^{-(i-1)\theta}\psi(z_i)e^{(i-1)\theta}
=(1+\beta z_i)^{-(i-1)}\psi(z_i)
$.
Since $e^{\theta}\ket{-r}=\ket{-r}$, we have
\begin{align*}
\Phi(z_1,\dots,z_r)
&=\bra{0}e^{H(x)}D_1D_2\dots D_r\ket{-r}\\
&=
\det
(
\bra{0}
e^{H(x)}
D_i
e^{-H(x)}
\psi^\ast_{-j}
\ket{0}
)_{1\leq i,j\leq r}\\
&=
\det
\left(
(1+\beta z_i)^{-(i-1)}
(\textstyle \sum_{m=0}^{\infty}h_m(x)z_i^m)
\bra{0}
\psi(z_i)
\psi^\ast_{-j}
\ket{0}
\right)_{1\leq i,j\leq r}\\
&=
\det
\left(
(1+\beta z_i)^{-(i-1)}
(\textstyle \sum_{m=0}^{\infty}h_m(x)z_i^m)
z_i^{-j}
\right)_{1\leq i,j\leq r}.
\end{align*}
Comparing the coefficients of $z_1^{\lambda_1-1}\cdots z_r^{\lambda_r-r}$ on the both sides gives the desired expression.
\end{proof}

\section{Application 1: $G_\lambda(x)$-expansion of symmetric functions}\label{sec:app1}

In the following two sections, we present a new method of deriving Pieri type formulas for $K$-theoretic polynomials.
We will define an action of non-commutative Schur polynomials~\cite{fomin1998noncommutative} on Grothendieck polynomials and dual stable Grothendieck polynomials by using their free-fermionic presentations.
This enables us to express symmetric functions of the form $s_\lambda(x)G_\mu(x)$ (\textit{resp.}~$s_\lambda(x)g_\mu(x)$) as a linear combination of Grothendieck polynomials (\textit{resp.}~dual stable Grothendieck polynomials).

\subsection{$\beta$-twisted Schur operators}
Let 
\[
\mathfrak{X}:=\bigoplus_{\lambda}\QQ[\beta]\cdot \lambda
\]
be the $\QQ[\beta]$-module freely generated by all partitions $\lambda$.
We define a linear operator
$
u_i:\mathfrak{X}\to \mathfrak{X}
$, ($i>0$), which we will call a \textit{$\beta$-twisted Schur operator}.
For any sequence $n=(n_1,\dots,n_\ell)$, we let $\overline{n}=(\overline{n}_1,\dots,\overline{n}_\ell)$ denote the smallest partition that satisfies $n_i\leq \overline{n}_i$ for all $i$.
We have $\overline{n}_i=\max[n_i,n_{i+1},\dots,n_\ell]$.

We write $\ee_i=(0,\dots,\stackrel{\stackrel{i}{\vee}}{1},\dots,0)$.
\begin{defi}
Let $u_i:\mathfrak{X}\to \mathfrak{X}$ be the linear operator that acts on a partition $\lambda$ as
\begin{gather*}
u_i\cdot \lambda=(-\beta)^{\zet{\overline{\lambda+\ee_i}}-\zet{\lambda+\ee_i} }\cdot \overline{\lambda+\ee_i}.
\end{gather*}
\end{defi}
\begin{example}
\[
u_1\cdot \Tableau{ & \\ &}=\Tableau{ & & \\  &},\quad
u_2\cdot \Tableau{ & \\ &}=-\beta \cdot \Tableau{ & & \\ & &},\quad
u_3\cdot \Tableau{ & \\ &}=\Tableau{ & \\ & \\ \\}.
\]
\end{example}
\begin{example}
Since $\overline{(\overline{\lambda+\ee_i})+\ee_j}=\overline{\lambda+\ee_i+\ee_j}$ for $i<j$, the action of operators of the form $u_{i_1}\cdots u_{i_r}$ $(i_1>\cdots>i_r)$ is expressed as
\begin{equation}\label{eq:uuu}
u_{i_1}\cdots u_{i_r}\cdot \lambda=
(-\beta)^{\zet{\overline{\lambda+\ee_{i_1}+\cdots+\ee_{i_r} }}-
\zet{\lambda
{\color{red}+\ee_{i_1}+\cdots+\ee_{i_r}}
}}
\cdot
\overline{\lambda+\ee_{i_1}+\cdots+\ee_{i_r}}.
\end{equation}
\end{example}

\begin{lemma}\label{lemma:knuth}
The $\beta$-twisted Schur operators satisfy the following commutative relations.
\begin{eqnarray}\label{eq:Knuth}
\begin{gathered}
u_iu_ku_j=u_ku_iu_j,\quad i\leq j<k,\\
u_ju_iu_k=u_ju_ku_i,\quad i<j\leq k.
\end{gathered}
\end{eqnarray}
\end{lemma}
\begin{proof}
They are directly checked by seeing their actions on the basis.
\end{proof}
Equation (\ref{eq:Knuth}) in Lemma \ref{lemma:knuth} are often called the \textit{Knuth relation}.
It was proved by Fomin and Greene~\cite{fomin1998noncommutative} that the theory of non-commutative Schur functions is applicable to any set of operators that satisfies the Knuth relation.

\subsection{Non-commutative Schur functions}

Let $T$ be a \textit{semi-standard tableau}, or a \textit{tableau} \cite{fulton_1996}.
The \textit{column word} $w_T$ of $T$ is the sequence of numbers obtained by reading the entries of $T$ from bottom to top in each column, starting in the left column and moving to the right.
For example, $w_T=3215344$ for $T=\Tableau{1&3&4&4\\2&5\\3}$.
We define the monomial $u^T$ 
as
\[
u^T:=u_{w_T(1)}u_{w_T(2)}\dots u_{w_T(N)}.
\]

\begin{defi}[Non-commutative Schur function]\label{defi:noncommutativeSchur}
For a partition $\lambda$, we define $s_\lambda(u_1,\dots,u_n)$ by
\begin{gather*}
s_\lambda(u_1,\dots,u_n):=\sum_{
\substack{
T\, \mathrm{is}\,\mathrm{of}\, \mathrm{shape}\, \lambda,\\ 
\mathrm{Each}\, \mathrm{entry}\, \mathrm{of}\, T\, \mathrm{is}\, \leq n.}
}
{\hspace{-15pt}u^T}.
\end{gather*}
\end{defi}

If $\lambda=(1^i)=(\overbrace{1,\dots,1}^i)$, the operator $e_i(u_1,\dots,u_n):=s_{(1^i)}(u_1,\dots,u_n)$ is called the \textit{$i$-th non-commutative elementary symmetric polynomial}.
If $\lambda=(i)$, $h_i(u_1,\dots,u_n):=s_{(i)}(u_1,\dots,u_n)$ is called the \textit{$i$-th non-commutative complete symmetric polynomial}.

The following proposition is given by Fomin-Greene~\cite{fomin1998noncommutative}.
\begin{prop}[Fundamental properties of non-commutative Schur functions]\label{prop:noncommutativeSchur}
Let $u_1,\dots,u_n$ be a set of non-commutative operators with the Knuth relation $(\ref{eq:Knuth})$.
Let $\Lambda_n(u)$ denote the ring of non-commutative Schur functions in $u_1,\dots,u_n$.
Then $\Lambda_n(u)$ is commutative; that is, we have 
\begin{gather}
s_\lambda(u_1,\dots,u_n)s_\mu(u_1,\dots,u_n)
=
s_\mu(u_1,\dots,u_n)s_\lambda(u_1,\dots,u_n)
\end{gather}
for any $\lambda,\mu$.
Moreover, the $($commutative$)$ ring $\Lambda_n(u)$ is generated by all non-commutative elementary polynomials $e_i(u_1,\dots,u_n)$ for $i=0,1,\dots,n$.
As a polynomial of $e_i(u_1,\dots,u_n)$'s, the non-commutative Schur polynomial $s_\lambda(u_1,\dots,u_n)$ is expressed as
\[
s_\lambda(u_1,\dots,u_n)=\det(e_{\lambda_i-i+j}(u_1,\dots,u_n))_{1\leq i,j\leq r},\quad
\ell(\lambda)\leq n\leq r,
\]
which is exactly same as the Jacobi-Trudi formula for ordinal Schur polynomials.
\end{prop}


\subsection{$u_n$-action on Grothendieck polynomials}

Let $\pi:\mathfrak{X}\to \widehat{\Lambda}$ is the $\QQ[\beta]$-linear map that sends a partition $\lambda$ to $G_\lambda(x)$.
The following proposition provides an algorithm to express a product $s_\lambda(x)G_\mu(x)$ as a linear combination of Grothendieck polynomials.
\begin{prop}\label{prop:product_sG}
For $\ell(\lambda),\ell(\mu)\leq r$, the equation
\begin{gather*}
s_\lambda(x)G_\mu(x)\equiv 
\pi(s_\lambda(u_1,\dots,u_r)\cdot \mu)
\mod{M_{r+1}}
\end{gather*}
holds.
In other words, we have
\[
s_\lambda(x_1,\dots,x_n)G_\mu(x_1,\dots,x_n)
=
\pi(s_\lambda(u_1,\dots,u_r)\cdot \mu)\vert_{x_{n+1}=x_{n+2}=\cdots=0}
\]
for $\ell(\lambda),\ell(\mu)\leq n\leq r$.
\end{prop}
\begin{proof}
From Proposition \ref{prop:noncommutativeSchur}, it suffices to prove
\[
e_i(x)G_\mu(x)\equiv 
\pi(e_i(u_1,\dots,u_r)\cdot \mu)
\mod{M_{r+1}}
\]
for $i=0,1,\dots,r$.
Write
$
f(x)=
\bra{0}
e^{H(x)}
\psi_{n_1-1}e^{\Theta}
\cdots
\psi_{n_r-r}e^{\Theta}
\ket{-r}
$ 
for a sequence of integers $n_1,\dots,n_r$.
Since $e^{H(x)}a_{-i}e^{-H(x)}=a_{-i}+p_i(x)$, $[a_{-i},\psi_{m}]=\psi_{m+i}$, and $[a_{-i},e^{\Theta}]=0$, we have
\begin{align*}
p_i(x)f(x)
&=
\bra{0}
e^{H(x)}
a_{-i}
\psi_{n_1-1}e^{\Theta}
\cdots
\psi_{n_r-r}e^{\Theta}
\ket{-r}\\
&=
\sum_{j=1}^{r}
\bra{0}
e^{H(x)}
\psi_{n_1-1}e^{\Theta}
\cdots
\psi_{n_j-j+i}e^{\Theta}
\cdots
\psi_{n_r-r}e^{\Theta}
\ket{-r}\\
&\hspace{50pt}+
\bra{0}
e^{H(x)}
\psi_{n_1-1}e^{\Theta}
\cdots
\psi_{n_r-r}e^{\Theta}
a_{-i}
\ket{-r}.
\end{align*}
This equation implies
\begin{equation}\label{eq:p_if}
p_i(x)f(x)
\equiv
\sum_{j=1}^{r}
\bra{0}
e^{H(x)}
\psi_{n_1-1}e^{\Theta}
\cdots
\psi_{n_j-j+i}e^{\Theta}
\cdots
\psi_{n_r-r}e^{\Theta}
\ket{-r}\mod{M_{r+1}}
\end{equation}
because $\bra{0}
e^{H(x)}
\psi_{n_1-1}e^{\Theta}
\cdots
\psi_{n_r-r}e^{\Theta}
a_{-i}
\ket{-r}
$
is contained in $M_{r+1}$ by Lemma \ref{lemma:small_polynomial_1}.
Let $E_i(p_1,\dots,p_i)$ be the polynomial in $p_1,\dots,p_i$ that satisfies $$e_i(x)=E_i(p_1(x),\dots,p_i(x)).$$
From (\ref{eq:p_if}), we show that the product $e_i(x)f(x)$ satisfies
\begin{align*}
&e_i(x)f(x)
=
\bra{0}
e^{H(x)}
E_i(a_{-1},\dots,a_{-i})
\psi_{n_1-1}e^{\Theta}
\cdots
\psi_{n_r-r}e^{\Theta}
\ket{-r}\\
&\equiv\ \ 
\hbox to 0pt{\lower 5pt\hbox{\hspace{-13pt}$\displaystyle \mathop{}_{1\leq m_1<\dots<m_i\leq r}$}}
\sum\ \ 
\bra{0}
e^{H(x)}
\psi_{n_1-1}e^{\Theta}
\cdots
\psi_{n_{m_1}-m_1+1}e^{\Theta}
\cdots
\psi_{n_{m_i}-m_i+1}e^{\Theta}
\cdots
\psi_{n_r-r}e^{\Theta}
\ket{-r}.
\end{align*}
For any sequence $n=(n_1,\dots,n_r)$, write
\[
Y(n)=
\bra{0}e^{H(x)}
\psi_{n_1-1}e^{\Theta}
\cdots
\psi_{n_{r}-r}e^{\Theta}
\ket{-r}.
\]
(Note that $\pi(\lambda)=Y(\lambda)$ if $\lambda$ is a partition.)
Substituting $f(x)=G_\mu(x)$ to the above equation gives
\begin{align*}
&e_i(x)G_\mu(x)
\equiv 
\sum_{1\leq m_1<\dots<m_i\leq r}
Y(\mu+\ee_{m_1}+\dots+\ee_{m_i})
\mod{M_{r+1}}.
\end{align*}
From Corollary \ref{cor:X(n)} and (\ref{eq:uuu}), this implies
\begin{align*}
&e_i(x)G_\mu(x)\\
&\equiv
\sum_{1\leq m_1<\dots<m_i\leq r}
(-\beta)^{\zet{\overline{\mu+\ee_{m_1}+\cdots+\ee_{m_i}}}
-
\zet{\mu
{\color{red}+\ee_{m_1}+\cdots+\ee_{m_i}}
}}
Y(\overline{\mu+\ee_{m_1}+\cdots+\ee_{m_i}})\\
&=
\sum_{1\leq m_1<\dots<m_i\leq r}
\pi\left(
u_{m_i}\dots u_{m_1}\cdot \mu
\right)\\
&=\pi(e_{i}(u_1,\dots,u_r)\cdot \mu)
\mod{M_{r+1}}.
\end{align*}
\end{proof}

From Proposition \ref{prop:product_sG}, we find a systematic way to express symmetric polynomials of the form $s_\lambda(x_1,\dots,x_n)G_\mu(x_1,\dots,x_n)$ as a linear combination of $G_\lambda(x_1,\dots,x_n)$'s.
See the examples below.
\begin{example}
Since
\[
h_2(u_1,u_2)=s_{\scriptsize\Tableau{ & }}(u_1,u_2)=u_1u_1+u_1u_2+u_2u_2,
\]
we have 
\begin{gather*}
h_2(x_1,x_2)\cdot G_\emptyset(x_1,x_2)
=
G_{\scriptsize\Tableau{ & }}(x_1,x_2)
-\beta
G_{\scriptsize\Tableau{ & \\ \\}}(x_1,x_2)
+
\beta^2
G_{\scriptsize\Tableau{ & \\ &}}(x_1,x_2),
\\
h_2(x_1,x_2)
\cdot
G_{\scriptsize\Tableau{ \\}}(x_1,x_2)
=
G_{\scriptsize\Tableau{ & & }}(x_1,x_2)
+
G_{\scriptsize\Tableau{ & \\ \\}}(x_1,x_2)
-\beta
G_{\scriptsize\Tableau{ & \\ &}}(x_1,x_2),
\\
h_2(x_1,x_2)
\cdot
G_{\scriptsize\Tableau{ & }}(x_1,x_2)
=
G_{\scriptsize\Tableau{ & & & }}(x_1,x_2)
+
G_{\scriptsize\Tableau{ & & \\ \\}}(x_1,x_2)
+
G_{\scriptsize\Tableau{ & \\ &}}(x_1,x_2),\ \textit{etc.}
\end{gather*}
\end{example}

\begin{example}
Let $\lambda=\Tableau{ & \\ \\}$.
All Young tableaux of the shape $\lambda$ with entries at most $3$ are given as follows:
\[
\Tableau{1&1\\2}\quad
\Tableau{1&2\\2}\quad
\Tableau{1&3\\2}\quad
\Tableau{1&1\\3}\quad
\Tableau{1&2\\3}\quad
\Tableau{1&3\\3}\quad
\Tableau{2&2\\3}\quad
\Tableau{2&3\\3}.
\]
We therefore have
\begin{align*}
&s_{\scriptsize\Tableau{ & \\ \\}}(u_1,u_2,u_3)\\
&=
u_2u_1u_1+u_2u_1u_2+u_2u_1u_3+u_3u_1u_1+
u_3u_1u_2+u_3u_1u_3+u_3u_2u_2+u_3u_2u_3.
\end{align*}
By using this, we obtain $(s_\lambda=s_\lambda(x_1,x_2,x_3)$, $G_\lambda=G_\lambda(x_1,x_2,x_3))$
\begin{gather*}
s_{\scriptsize\Tableau{ & \\ \\}}{\color{red} G_\emptyset}
=
G_{\scriptsize\Tableau{ & \\ \\}}
-\beta
G_{\scriptsize\Tableau{ & \\ &}}
-2\beta
G_{\scriptsize\Tableau{ & \\ \\ \\}}
+2\beta^2
G_{\scriptsize\Tableau{ & \\ & \\ \\}}
-2{\color{red}\beta^3}
G_{\scriptsize\Tableau{ & \\ & \\ & \\}},
\\
s_{\scriptsize\Tableau{ & \\ \\}}
G_{\scriptsize\Tableau{ \\}}
=
G_{\scriptsize\Tableau{ & & \\ \\}}
+
G_{\scriptsize\Tableau{ & \\ & \\}}
+
G_{\scriptsize\Tableau{ & \\ \\ \\}}
-2\beta
G_{\scriptsize\Tableau{ & \\ & \\ \\}}
-\beta
G_{\scriptsize\Tableau{ & & \\ \\ \\}}
+2\beta^2
G_{\scriptsize\Tableau{ & \\ & \\ & \\}},\ \textit{etc.}
\end{gather*}
\end{example}

\section{Application 2: Pieri type formula for $g_\lambda$}\label{sec:app2}

Let ${d}_i:\mathfrak{X}\to \mathfrak{X}$ be the $\QQ[\beta]$-linear operator defined by
\begin{gather*}
{d}_i\cdot \lambda=
\begin{cases}
\lambda\cup \{\mbox{a box in $i$-th row} \}, & \mbox{if possible},\\
(-\beta)\cdot \lambda, & \mbox{otherwise}.
\end{cases}
\end{gather*}
By seeing their actions on the basis, we can check that the operators $d_1,d_2,\dots$ satisfy the Knuth relation:
\begin{lemma}\label{lemma:knuth2}
We have the following commutative relations.
\begin{eqnarray}\label{eq:Knuth2}
\begin{gathered}
d_id_kd_j=d_kd_id_j,\quad i\leq j<k,\\
d_jd_id_k=d_jd_kd_i,\quad i<j\leq k.
\end{gathered}
\end{eqnarray}
\end{lemma}
\begin{example}
\[
d_1\cdot \Tableau{ \\ \\}=\Tableau{ & \\ \\},\quad
d_3\cdot \Tableau{ \\ \\}=\Tableau{ \\ \\ \\},\quad
d_2\cdot \Tableau{ \\ \\}=
d_4\cdot \Tableau{ \\ \\}=
d_5\cdot \Tableau{ \\ \\}=
\cdots=
-\beta\cdot \Tableau{ \\ \\}.
\]
\end{example}

For any Young tableau $T$, we write
\[
d^T=d_{w_T(1)}d_{w_T(2)}\cdots d_{w_T(N)}.
\]
We also define the non-commutative Schur functions $s_\lambda(d_1,\dots,d_n)$ in the similar manner to $s_\lambda(u_1,\dots,u_n)$ (Definition \ref{defi:noncommutativeSchur}).

Let $\rho:\mathfrak{X}\to \Lambda$ be the $\QQ[\beta]$-linear map defined by
\[
\rho:\lambda\mapsto g_\lambda(x).
\]
The following proposition is an analogy of Proposition \ref{prop:product_sG}.
\begin{prop}\label{prop:product_sg}
For $\ell(\lambda)\leq s$, $\ell(\mu)\leq r$, we have 
\[
s^{(r+s-1)}_\lambda(x;-\beta)g_\mu(x)= \rho(s_\lambda(d_1,\dots,d_{r+s})\cdot \mu),
\]
where 
$s^{(r+s-1)}_\lambda(x;-\beta)=s_\lambda(\overbrace{-\beta,\dots,-\beta}^{r+s-1},x_1,x_2,\dots)$.
\end{prop}
\begin{proof}
From Proposition \ref{prop:noncommutativeSchur}, it suffices to prove
\begin{equation}\label{eq:elementary_and_g}
e^{(r+s-1)}_i(x;-\beta)g_\mu(x)= \rho(e_i(d_1,\dots,d_{r+s})\cdot \mu)
\end{equation}
for $0\leq i\leq s$.
Let 
$g(x)=
\bra{0}
e^{H(x)}
\psi_{n_1-1}e^{-\theta}
\cdots
\psi_{n_r-r}e^{-\theta}
\psi_{-r-1}e^{-\theta}
\cdots
\psi_{-r-s} 
\ket{-r-s}
$
for any sequence of integers $n_1,\dots,n_{r}$.
Since $[a_{-i},e^{-\theta}]=-(-\beta)^ie^{-\theta}$, we have
\begin{align*}
&p_i(x)g(x)\\
&=
\bra{0}
e^{H(x)}
a_{-i}
\psi_{n_1-1}e^{-\theta}
\cdots
\psi_{n_r-r}e^{-\theta}
\cdots
\psi_{-r-s} 
\ket{-r-s}\\
&=
\sum_{j=1}^{r+s}
\left\{
\bra{0}
e^{H(x)}
\psi_{n_1-1}e^{-\theta}
\cdots
\psi_{n_j-j+i}e^{-\theta}
\cdots
\psi_{-r-s} 
\ket{-r-s}
\right\}\\
&\hspace{10pt}
-(r+s-1)(-\beta)^i g(x)
+
\bra{0}
e^{H(x)}
\psi_{n_1-1}e^{-\theta}
\cdots
\psi_{n_r-r}e^{-\theta}
\cdots
\psi_{-r-s} 
a_{-i}
\ket{-r-s}.
\end{align*}
Note that the last term of this expression must vanish if $i\leq s$.
Moreover, because
\[
p^{(r+s-1)}_i(x;-\beta)=p_i(\overbrace{-\beta,\dots,-\beta}^{r+s-1},x)=\overbrace{(-\beta)^i+\cdots+(-\beta)^i}^{r+s-1}+x_1^i+x_2^i+\cdots,
\]
the equation can be simplified as
\[
p^{(r+s-1)}_i(x;-\beta)g(x)
=
\sum_{j=1}^{r+s}
\bra{0}
e^{H(x)}
\psi_{n_1-1}e^{-\theta}
\cdots
\psi_{n_j-j+i}e^{-\theta}
\cdots
\psi_{-r-s} 
\ket{-r-s}.
\]
From this, we find that the product $e^{(r+s-1)}_i(x;-\beta)g(x)$ satisfies the following equation:
\begin{align*}
&e^{(r+s-1)}_i(x;-\beta)g(x)\\
&=
\hbox to 0pt{\hspace{-13pt}\lower 5pt\hbox{$\displaystyle\mathop{}_{1\leq m_1<\dots<m_i\leq r+s}$}}
\sum\ \ 
\bra{0}
e^{H(x)}
\psi_{n_1-1}e^{-\theta}
\cdots
\psi_{n_{m_1}-m_1+1}e^{-\theta}
\cdots
\psi_{n_{m_i}-m_i+1}e^{-\theta}
\cdots
\psi_{-r-s} 
\ket{-r-s}.
\end{align*}
Substituting $g(x)=g_\mu(x)$ to this equation gives (\ref{eq:elementary_and_g}).
\end{proof}

\begin{example}
Let $s=1$ and $r=0$.
In this case $s^{(0)}_{(n)}(x;-\beta)=h_n(x)$.
Since 
$h_n(d_1)\cdot \emptyset=d_1^n\cdot \emptyset=(n)$,
we have
$
h_n(x)=g_{(n)}(x)
$.
\end{example}

\begin{example}
Let $s=n$ and $r=0$.
In this case $s^{(n-1)}_{(1^n)}(x;-\beta)=e^{(n-1)}_n(x;-\beta)
=\sum_{j=0}^{n}(-\beta)^j\binom{n-1}{j}e_{n-j}(x)
$.
Since 
$e_n(d_1,\dots,d_n)\cdot \emptyset
=
d_n\cdots d_2d_1\cdot \emptyset
=(1^n)$,
we have
$
\sum_{j=0}^{n}(-\beta)^j\binom{n-1}{j}e_{n-j}(x)=g_{(1^n)}(x)
$.
\end{example}

\subsection{Pieri type formula for $e_i(x)g_\lambda(x)$}

Proposition \ref{prop:product_sg} is unfortunately a bit complicated as it contains a function of the form $s^{(r+s-1)}_\lambda(x;-\beta)$.
However, when the partition $\lambda$ is relatively simple (for example, if $\lambda=(1^i)$ or $\lambda=(i)$), we can handle the situation.
In the following, we consider the expansions of symmetric functions of the form $e_i(x)g_\lambda(x)$ and $h_i(x)g_\lambda(x)$.

To calculate the product $e_i(x)g_\lambda(x)$, it is convenient to introduce the formal power series
$
E(t)=\sum_{i=0}^{\infty}e_i(x)t^i=\prod_{j=1}^\infty(1+x_jt)
$.
From the expression $e^{(p)}_i(x;-\beta)=e_i(\overbrace{-\beta,\dots,-\beta}^p,x_1,x_2,\dots)$, we find
\[
\sum_{i=0}^{\infty}e^{(p)}_i(x;-\beta)t^i=(1-\beta t)^{p}E(t).
\]
Therefore, from Proposition \ref{prop:product_sg}, we have
\begin{equation}\label{eq:rinji}
E(t)g_\lambda(x)\equiv
\rho\left(
(1-\beta t)\cdot 
\left(
\frac{1+d_{r+s}t}{1-\beta t}
\cdots 
\frac{1+d_2t}{1-\beta t}\cdot
\frac{1+d_1t}{1-\beta t}
\right)\cdot \lambda
\right)\mod{t^{s+1}}
\end{equation}
for $\ell(\lambda)\leq r$.
Taking the limit as $r,s\to \infty$ gives the formal expression
\begin{equation}\label{eq:rinji2}
E(t)g_\lambda(x)=
\rho\left(
(1-\beta t)\cdot 
\left(
\cdots \cdot
\frac{1+d_2t}{1-\beta t}\cdot
\frac{1+d_1t}{1-\beta t}
\right)\cdot \lambda
\right).
\end{equation}
For example, since
\begin{gather*}
(1-\beta t)\cdot 
\left(
\cdots 
\frac{1+d_3t}{1-\beta t}\cdot
\frac{1+d_2t}{1-\beta t}\cdot
\frac{1+d_1t}{1-\beta t}
\right)\cdot \emptyset
=
\emptyset
+
\Tableau{ \\}\frac{t}{1-\beta t}
+
\Tableau{ \\ \\}\frac{t^2}{(1-\beta t)^2}
+\cdots,
\\
\begin{split}
&(1-\beta t)\cdot 
\left(
\cdots 
\frac{1+d_3t}{1-\beta t}\cdot
\frac{1+d_2t}{1-\beta t}\cdot
\frac{1+d_1t}{1-\beta t}
\right)\cdot \Tableau{ \\}\\
&=
\left(
\Tableau{ \\}\frac{1}{1-\beta t}
+
\Tableau{ \\ \\}\frac{t}{(1-\beta t)^2}
+
\Tableau{ \\ \\ \\}\frac{t^2}{(1-\beta t)^3}
+
\cdots
\right)\\
&\hspace{30pt}
+
\left(
\Tableau{ & }\frac{t}{1-\beta t}
+
\Tableau{ & \\ \\}\frac{t^2}{(1-\beta t)^2}
+
\Tableau{ & \\ \\ \\}\frac{t^3}{(1-\beta t)^3}
+
\cdots
\right),
\end{split}
\end{gather*}
we have
\begin{gather*}
E(t)=
g_\emptyset
+
g_{\scriptsize\Tableau{ \\}}t
+
(
\beta 
g_{\scriptsize\Tableau{ \\}}
+
g_{\scriptsize\Tableau{ \\ \\}}
)t^2
+
(
\beta^2
g_{\scriptsize\Tableau{ \\}}
+
2\beta
g_{\scriptsize\Tableau{ \\ \\}}
+
g_{\scriptsize\Tableau{ \\ \\ \\}}
)t^3
+\cdots,
\\
\begin{split}
&E(t)g_{\scriptsize\Tableau{ \\}}
=
g_{\scriptsize\Tableau{ \\}}
+
(\beta
g_{\scriptsize\Tableau{ \\}}
+
g_{\scriptsize\Tableau{ \\ \\}}
+
g_{\scriptsize\Tableau{ & }}
)t\\
&\hspace{30pt}
+
(\beta^2
g_{\scriptsize\Tableau{ \\}}
+2\beta
g_{\scriptsize\Tableau{ \\ \\}}
+
\beta
g_{\scriptsize\Tableau{ & }}
+
g_{\scriptsize\Tableau{ \\ \\ \\}}
+
g_{\scriptsize\Tableau{ & \\ \\}}
)t^2
+\cdots.
\end{split}
\end{gather*}

\subsection{Pieri type formula for $h_i(x)g_\lambda(x)$}

Let $H(t):=\sum_{i=0}^{\infty}h_i(x)t^i=\prod_{j=1}^\infty\frac{1}{1-{\color{red}x_j}t}$ be the generating function of $h_i(x)$.
Since $h_i^{(p)}(x;-\beta)=h_i(\overbrace{-\beta,\dots,-\beta}^p,x_1,x_2,\dots)$, we have
\[
\sum_{i=0}^{\infty}h^{(p)}_i(x;-\beta)t^i=(1+\beta t)^{-p}H(t).
\]
Therefore, from Proposition \ref{prop:product_sg}, we have
\begin{equation}\label{eq:inftyexpress}
H(t)g_\lambda(x)=\rho\left(
\frac{1}{1+\beta t}
\frac{1+\beta t}{1-d_1t}
\frac{1+\beta t}{1-d_2t}
\cdots
\cdot \lambda
\right),
\end{equation}
where $\frac{1}{1-d_it}=1+d_it+d_i^2t^2+\cdots$.
We note that the expression (\ref{eq:inftyexpress}) does not cause a confusion because
$
\frac{1+\beta t}{1-d_it}\cdot \lambda=\lambda
$
for $i>\ell(\lambda)+1$.
For example, we have
\begin{gather*}
\frac{1}{1-d_1t}\cdot \Tableau{ & }
=
\Tableau{ & }
+
\Tableau{ &  & }t
+
\Tableau{ & & &}t^2
+\cdots,
\\
\frac{1+\beta t}{1-d_2t}\cdot \Tableau{ & }=
\Tableau{ & }
+
\left(
\beta \Tableau{ & }
+
\Tableau{  & \\ \\}
\right)t
+
\left(
\beta \Tableau{ & \\ \\}
+
\Tableau{  & \\ & \\}
\right)t^2, \ \textit{etc.}
\end{gather*}
Therefore we obtain
\begin{gather*}
H(t)=
g_\emptyset
+
g_{\scriptsize\Tableau{ \\}}t
+
g_{\scriptsize\Tableau{ & }}t^2
+
g_{\scriptsize\Tableau{ & & }}t^3
+\cdots,
\\
H(t)g_{\scriptsize\Tableau{ \\}}
=
g_{\scriptsize\Tableau{ \\}}
+
(\beta 
g_{\scriptsize\Tableau{ \\}}
+
g_{\scriptsize\Tableau{ & }}
{\color{red}
+
g_{\scriptsize\Tableau{ \\  \\}}
}
)t
+
(\beta 
g_{\scriptsize\Tableau{ & }}
+
g_{\scriptsize\Tableau{ & & }}
+
g_{\scriptsize\Tableau{ & \\ \\}}
)t^2
+\cdots,
\\
\begin{split}
&H(t)g_{\scriptsize\Tableau{ & }}
=
g_{\scriptsize\Tableau{ & }}
+
(\beta 
g_{\scriptsize\Tableau{ & }}
+
g_{\scriptsize\Tableau{ & \\ \\}}
+
g_{\scriptsize\Tableau{ & & }}
)t\\
&\hspace{30pt}
+
(\beta 
g_{\scriptsize\Tableau{ & & }}
+
\beta 
g_{\scriptsize\Tableau{ & \\ \\ }}
+
g_{\scriptsize\Tableau{ & & & }}
+
g_{\scriptsize\Tableau{ & & \\ \\}}
+
g_{\scriptsize\Tableau{ & \\ & \\}}
)t^2
+\cdots.
\end{split}
\end{gather*}

\section*{Acknowledgments}

This work is partially supported by JSPS Kakenhi Grant Number 19K03605.


\providecommand{\bysame}{\leavevmode\hbox to3em{\hrulefill}\thinspace}
\providecommand{\MR}{\relax\ifhmode\unskip\space\fi MR }
\providecommand{\MRhref}[2]{%
  \href{http://www.ams.org/mathscinet-getitem?mr=#1}{#2}
}
\providecommand{\href}[2]{#2}

\end{document}